\documentclass{amsart}

\usepackage{amsmath,amsfonts,amsthm,amssymb,verbatim,enumerate,graphicx,mathtools}
\usepackage[flushmargin]{footmisc}
\newtheorem{theorem}{Theorem}[section]
\newtheorem*{theorema}{Theorem A}
\newtheorem*{theoremb}{Theorem B}
\newtheorem*{theoremc}{Theorem C}
\newtheorem{lemma}[theorem]{Lemma}
\newtheorem{proposition}[theorem]{Proposition}

\newtheorem*{definitionnonum}{Definition}

\newtheorem*{remark}{Remark}
\numberwithin{equation}{section}
\def \isnatural {\in\mathbb{N}}

\newcommand{\tef}{transcendental entire function}
\newcommand\qfor{\quad\text{for }}

\newcommand \N{\mathbb{N}}
\newcommand \R{\mathbb{R}}

\newcommand{\hull}{\operatorname{hull}}
\newcommand{\hullo}{\operatorname{inn}}

\makeatletter
\def\blfootnote{\xdef\@thefnmark{}\@footnotetext}
\makeatother
\begin{document}
%
%
%
%
\title[The size and topology of quasi-Fatou components]{The size and topology of quasi-Fatou components of quasiregular maps}
\author{Daniel A. Nicks, \, David J. Sixsmith}
\address{School of Mathematical Sciences \\ University of Nottingham \\ 
NG7 2RD \\ UK}
\email{Dan.Nicks@nottingham.ac.uk}
\address{School of Mathematical Sciences \\ University of Nottingham \\ 
NG7 2RD \\ UK}
\email{David.Sixsmith@open.ac.uk}
%
%
%
%
\begin{abstract}
We consider the iteration of quasiregular maps of transcendental type from $\R^d$ to $\R^d$. In particular we study quasi-Fatou components, which are defined as the connected components of the complement of the Julia set.

Many authors have studied the components of the Fatou set of a {\tef}, and our goal in this paper is to generalise some of these results to quasi-Fatou components. First, we study the number of complementary components of quasi-Fatou components, generalising, and slightly strengthening, a result of Kisaka and Shishikura. Second, we study the size of quasi-Fatou components that are bounded and have a bounded complementary component. We obtain results analogous to those of Zheng, and of Bergweiler, Rippon and Stallard. These are obtained using techniques which may be of interest even in the case of {\tef}s.
\end{abstract}
\maketitle
%
%
%
%
\blfootnote{2010 \itshape Mathematics Subject Classification. \normalfont Primary 37F10; Secondary 30C65, 30D05.}
\blfootnote{Both authors were supported by Engineering and Physical Sciences Research Council grant EP/L019841/1.}
\section{Introduction}
In the study of complex dynamics, the first example of a {\tef} with a multiply connected Fatou component was given by Baker \cite{MR0153842}; see the survey \cite{MR1216719} for definitions and further background on complex dynamics. Since Baker's work many authors have studied multiply connected Fatou components; examples include the papers \cite{MR0419759, MR759304, MR3149847, MR3145128, MR2178719}. In this paper our goal is to see how some of these results can be extended to a class of maps on $\R^d$.

First we discuss these results from complex dynamics in more detail. Suppose that $f$ is a {\tef}. If $U_0$ is a Fatou component of $f$, then we adopt the standard notation by letting $U_k$ be the Fatou component of $f$ containing $f^k(U_0)$, for $k\isnatural$. Three important properties of the sequence $(U_k)_{k\isnatural}$ are as follows. The first is a result of Baker \cite[Theorem 3.1]{MR759304}. Here we denote the Euclidean distance from a point $x$ to a set $U \subset \R^d$ by $\operatorname{dist}(x,U) := \inf_{y\in U} |x - y|$.
\begin{theorema}
Suppose that $f$ is a {\tef}, and that $U_0$ is a multiply connected Fatou component of $f$. Then each $U_k$ is bounded and multiply connected, $U_k$ lies in a bounded component of the complement of $U_{k+1}$ for large $k$, and $\operatorname{dist}(0,U_k)\rightarrow\infty$ as $k\rightarrow\infty$.
\end{theorema}

We identify $\mathbb{C}$ with $\R^2$ in the obvious way. If $U \subset \R^d$ is a domain, then we denote the number of components of $\widehat{\R^d} \setminus U$ by $\operatorname{cc}(U)$, where $\widehat{\R^d}$ is the one-point compactification $\widehat{\R^d} := \R^d \cup \{ \infty\}$. In the plane this quantity is known as the \emph{connectivity} of $U$. Kisaka and Shishikura studied the connectivity of $U_k$, for large values of $k$. Their result \cite[Theorem A]{MR2458806} is as follows. 
\begin{theoremb}
Suppose that $f$ is a {\tef}, and that $U_0$ is a Fatou component of $f$. Then there exists $\nu\in\{1, 2, \infty\}$ such that $\operatorname{cc}(U_k) = \nu$, for all sufficiently large $k$. If $\nu = 1$, then $\operatorname{cc}(U_0) = 1$. If $\nu=2$, then $f : U_k \to U_{k+1}$ is a covering of annuli for all sufficiently large $k$.
\end{theoremb}
Note that the hypotheses of Theorem~B include both simply and multiply connected Fatou components. Returning now specifically to multiply connected Fatou components, Zheng \cite{MR2178719} studied the size of these domains. He showed that, for all sufficiently large $k$, the set $U_k$ contains a ``large'' annulus centred at the origin. Specifically he proved the following, which is a special case of the theorem of \cite{MR2178719}.
\begin{theoremc}
Suppose that $f$ is a {\tef}, and that $U_0$ is a multiply connected Fatou component of $f$. Then there exist sequences of real numbers $(r_k)_{k\isnatural}$ and $(R_k)_{k\isnatural}$ such that:
\begin{enumerate}[(a)]
\item $r_k \rightarrow\infty$ as $k\rightarrow\infty$;
\item $\lim_{k\rightarrow\infty} \frac{R_k}{r_k} = \infty$; and
\item $\{ x : r_k < |x| < R_k \} \subset U_k$, for all sufficiently large $k$.
\end{enumerate}
\end{theoremc}
Zheng's result was significantly strengthened in \cite{MR3149847}, where, amongst other things, it was shown that (b) above could be replaced by $$\liminf_{k\rightarrow\infty} \frac{\log R_k}{\log r_k} > 1.$$ 

With the aim of extending the above theory to more than two (real) dimensions, we now suppose that $d \geq 2$, and that $f : \R^d \to \R^d$ is a quasiregular map of transcendental type. We defer a full definition of a quasiregular map to Section~\ref{Sdefs}; for now we note that these maps are the natural generalization to higher dimensions of analytic maps of the plane. A quasiregular map is said to be of \emph{transcendental type} if it has an essential singularity at infinity. 

We need a different definition of the Julia set to that used in complex dynamics, so, following \cite{MR3009101, MR3265283}, we define the \emph{Julia set $J(f)$} to be the set of all $x \in \R^d$ such that
\begin{equation}
\label{Juliadef}
\operatorname{cap} \left(\R^d\backslash \bigcup_{k=1}^\infty f^k(U)\right) = 0,
\end{equation}
for every neighbourhood $U$ of $x$. We define the notation used in (\ref{Juliadef}) in Section~\ref{Sdefs}; for now we note that if $S\subset \R^d$ and cap$(S)=0$, then $S$ is, in a precise sense, a small set. A set $E\subset\R^d$ is said to be \emph{completely invariant} if $x \in E$ implies that $f(x) \in E$ and $f^{-1}(x) \subset E$. It follows from the definition that the Julia set $J(f)$ is closed and completely invariant. If $f$ is a quasiregular map of transcendental type, then  by \cite[Theorem~1.1]{MR3265283} we have card$(J(f)) =\infty$. Note that it follows from \cite[Theorem~1.2]{MR583633}, which is the quasiregular analogue of Picard's great theorem, that $J(f)$ is unbounded.

Following \cite{SixsmithNicks1, SixsmithNicks2}, we define the \emph{quasi-Fatou set $QF(f)$} of $f$ as the complement in $\R^d$ of the Julia set $J(f)$. Thus the quasi-Fatou set is a completely invariant open set which, if non-empty, has the Julia set as its boundary. Note that, by \cite[Theorem~1.2]{MR3265283}, if $f$ is a {\tef}, then $QF(f)$ is equal to the usual Fatou set. We stress that there are no assumptions about the normality of the family of iterates $(f^k)_{k\geq 0}$ in a \emph{quasi-Fatou component}, which is a connected component of the quasi-Fatou set. We retain the convention that if $U_0$ is a quasi-Fatou component of $f$, then $U_k$ is the quasi-Fatou component of $f$ containing $f^k(U_0)$, for $k\isnatural$.

If $G \subset\R^d$ is a domain with a bounded complementary component, then we say that $G$ is \emph{hollow}; otherwise we say that it is \emph{full}. In \cite[Theorem 1.3(a)]{SixsmithNicks1} it was shown that the conclusions of Theorem~A all hold, with ``multiply connected'' interpreted as ``hollow'', if $f$ is a quasiregular map of transcendental type, and $U_0$ is a quasi-Fatou component of $f$ which is bounded and hollow; see Lemma~\ref{Lboundednottopconv} below. We extend this study by giving results analogous to Theorem~B and Theorem~C.

In fact we are able to give an almost complete generalisation of Theorem~B.
\begin{theorem}
\label{theo:boundedhollowQFsimple}
Suppose that $f : \R^d \to \R^d$ is a quasiregular map of transcendental type, and that $U_0$ is a quasi-Fatou component of $f$. Then $\operatorname{cc}(U_{k})\leq \operatorname{cc}(U_{k-1})$, for $k\isnatural$. If $\operatorname{cc}(U_0) \in \{1, \infty\}$, then $\operatorname{cc}(U_k)=\operatorname{cc}(U_0)$, for $k\isnatural.$ On the other hand, if $1 < \operatorname{cc}(U_0) < \infty$, then $\operatorname{cc}(U_k)=2$, for all sufficiently large values of $k$.
\end{theorem}
\begin{remark}\normalfont
We note that this result is, in fact, slightly stronger than Theorem~B (apart from the final conclusion about a ``covering of annuli''). We observe also that Kisaka and Shishikura used the Riemann-Hurwitz formula in the proof of their result. Our proof is, in fact, almost entirely topological, and so gives an alternative route to most of Theorem B.
\end{remark}
There is no example of a quasiregular map of transcendental type with a quasi-Fatou component $U$ that is known to be unbounded and hollow. However, it is known \cite[Theorem 1.4]{SixsmithNicks1} that in this case $U$ would be completely invariant and would have no unbounded complementary components. In the remainder of this paper we study quasiregular maps of $\R^d$, of transcendental type, with a bounded hollow quasi-Fatou component, and give results analogous to Theorem~C. In order to give a full statement of our results, we require a number of preliminary definitions. Note that all boundaries and closures in this paper are taken in $\R^d$.
\begin{definitionnonum}\normalfont
Suppose that $U \subset \R^d$ is a bounded domain. 
\begin{itemize}
\item We let $\hull(U)$ denote the \emph{topological hull} of the domain, in other words the union of $U$ with its bounded complementary components.
\item We let $\partial_{out} U$ denote the \emph{outer boundary} of $U$, in other words $\partial_{out} U :=~\partial \hull(U)$.
\item If the origin lies in a bounded complementary component of $U$, then we let $\hullo(U)$ denote this \emph{inner complementary component}.
\item If the origin lies in a bounded complementary component of $U$, then we let $\partial_{inn} U$ denote the \emph{inner boundary} of $U$, in other words $\partial_{inn} U :=~\partial \hullo(U)$.
\end{itemize}
\end{definitionnonum}
We note that, in the last two definitions, we are not claiming a special role for the origin, the choice of which is both convenient and standard (see, for example, \cite[p.1265]{MR3149847}). Indeed our results are unchanged if, in these definitions, the origin is replaced by some other point in $\R^d$.

Our next result gives a precise statement of the idea that the inner and outer boundaries of bounded hollow quasi-Fatou components are, after sufficiently many iterates, always ``far apart''.
\begin{theorem}
\label{theo:curvefambetterDan}
Suppose that $f : \R^d \to \R^d$ is a quasiregular map of transcendental type, and that $U_0$ is a bounded hollow quasi-Fatou component of $f$. Then the origin lies in a bounded complementary component of $U_k$, for all large $k$, and
$$
\lim_{k\rightarrow\infty} \frac{1}{k} \log \log \frac{\inf \{ |x| : x \in \partial_{out} U_k \}}{\sup \{|x| : x \in \partial_{inn} U_k\}} = \infty.
$$
\end{theorem}
We obtain the following corollary of Theorem~\ref{theo:curvefambetterDan}. This is a generalisation and strengthening of Theorem~C, in the case of hollow quasi-Fatou components with \emph{finitely} many complementary components.
\begin{theorem}
\label{theo:curvefambetter}
Suppose that $f : \R^d \to \R^d$ is a quasiregular map of transcendental type, and that $U_0$ is a hollow quasi-Fatou component of $f$ with finitely many complementary components. Then there exist sequences of real numbers $(r_k)_{k\isnatural}$ and $(R_k)_{k\isnatural}$ such that:
\begin{enumerate}[(a)]
\item $r_k \rightarrow\infty$ as $k\rightarrow\infty$;\label{theointoinf}
\item $\lim_{k\rightarrow\infty} \frac{1}{k} \log \log \frac{R_k}{r_k} = \infty$;\label{theofracsgrow} and
\item $\{ x : r_k < |x| < R_k \} \subset U_k$, for all sufficiently large $k$.
\end{enumerate}
\end{theorem}
For a bounded hollow quasi-Fatou component with \emph{infinitely} many complementary components, Theorem~\ref{theo:curvefambetterDan} gives no information about the size of the quasi-Fatou components, in view of the infinitely many ``holes'' between the inner and outer boundaries. Our final result shows that, in a precise sense, the quasi-Fatou components are eventually of very large measure. Here if $G \subset \R^d$ is a domain, then we denote the $d$-dimensional Lebesgue measure of $G$ by meas$(G)$. Also, if $f : \R^d \to \R^d$ is quasiregular, then we define the \emph{maximum modulus function} by $$M(R,f) := \max_{|x| = R} |f(x)|, \qfor R>0.$$ We also let $M^k(R, f)$ denote the $k$th iterate of $M(R, f)$ with respect to the first variable, in other words $M^k (R,f) = M(M^{k-1}(R,f),f)$, for $k > 1$.
\begin{theorem}
\label{theo:measureofboundedhollowQF}
Suppose that $f : \R^d \to \R^d$ is a quasiregular map of transcendental type, that $U_0$ is a bounded hollow quasi-Fatou component of $f$, and that $R>0$. Then there exist $k_0, \ \ell \in \N$ and $\alpha > 0$ such that
\begin{equation}
\label{eq:tnewer}
\operatorname{meas}(U_{k+\ell}) \geq \alpha^k \left(M^k(R, f)\right)^d, \qfor k \geq k_0.
\end{equation}
It follows that
\begin{equation}
\label{eq:tnew}
\lim_{k\rightarrow\infty} \frac{1}{k}\log\log(\operatorname{meas}(U_k)) = \infty.
\end{equation}
\end{theorem}

The structure of this paper is as follows. First, in Section~\ref{Sdefs} we recall various definitions, such as those of quasiregularity and the capacity of a condenser, and give some known results required in the rest of the paper. In Section~\ref{S.topology} we give some topological results, which are prerequisite to the proof of Theorem~\ref{theo:boundedhollowQFsimple}. We prove Theorem~\ref{theo:boundedhollowQFsimple} in Section~\ref{S:boundedhollowQFsimple}. Next, in Section~\ref{S:curvefambetterDan}, we prove Theorem~\ref{theo:curvefambetterDan} and Theorem~\ref{theo:curvefambetter}. Finally, in Section~\ref{S:measureofboundedhollowQF} we prove Theorem~\ref{theo:measureofboundedhollowQF}.
%
%
%
%
%
%
%
%
\section{Definitions and background results}
\label{Sdefs}
\subsection{Quasiregular maps}
We refer to \cite{MR1238941, MR950174} for a detailed treatment of quasiregular maps. Here we recall some definitions and properties used in this paper.

Suppose that $d\geq 2$, that $G \subset \R^d$ is a domain, and that $1 \leq p < \infty$. The \emph{Sobolev space} $W^1_{p,loc}(G)$ consists of those functions $f : G \to \R^d$ for which all first order weak partial derivatives exist and are locally in $L^p$. We say that $f$ is \emph{quasiregular} if $f \in W^1_{d,loc}(G)$ is continuous, and there exists $K_O \geq 1$ such that
\begin{equation}
\label{KOeq}
|D f(x)|^d \leq K_O J_f(x) \quad a.e.
\end{equation}
Here $D f(x)$ denotes the derivative,
$$
|D f(x)| := \sup_{|h|=1} |Df(x)(h)|
$$
is the norm of the derivative, and $J_f (x)$ denotes the Jacobian determinant. We also define
$$
\ell(Df(x)) := \inf_{|h|=1} |Df(x)(h)|.
$$

If $f$ is quasiregular, then there also exists $K_I \geq 1$ such that 
\begin{equation}
\label{KIeq}
K_I \ell(D f(x))^d \geq J_f(x) \quad a.e.
\end{equation}
The smallest constants $K_O$ and $K_I$ for which (\ref{KOeq}) and (\ref{KIeq}) hold are denoted by $K_O (f)$ and $K_I (f)$. 

If $f$ and $g$ are quasiregular maps, and $f$ is defined in the range of $g$, then $f \circ g$ is quasiregular and
\begin{equation}
\label{Keq}
K_I(f \circ g) \leq K_I(f)K_I(g),
\end{equation}
by \cite[Theorem II.6.8]{MR1238941}.

Many properties of holomorphic functions extend to quasiregular maps; we frequently use the fact that non-constant quasiregular maps are discrete and open. If $f$ is a discrete open map, then we denote by $B_f$ the \emph{branch set} of $f$; in other words the set of points $x$ such that $f$ is not a local homeomorphism at $x$. We use the following \cite{MR0172256, MR0220254}, \cite{MR0200928}, which was first proved by {\v{C}}ernavski{\u\i}. 
\begin{lemma}
\label{lemma:branch}
Suppose that $f : \R^d \to \R^d$ is a discrete open map. Then the topological dimension of $B_f$ and $f(B_f)$ is at most $d-2$.
\end{lemma}

If $f : \R^d \to \R^d$ and $A \subset \R^d$, then we denote the \emph{maximum multiplicity} of $f$ in $A$ by 
$$
N(f, A) := \sup_{y\in\R^d} \operatorname{card}(A \cap f^{-1}(y)),
$$
and we denote the \emph{local index} at a point $x$ by 
$$
i(x,f) := \inf_U \left(\sup_{y\in\R^d} \operatorname{card}(U \cap f^{-1}(y))\right),
$$
where the infimum is taken over all neighbourhoods $U$ of $x$. We use the fact \cite[Proposition I.4.10]{MR1238941} that if $f : \R^d \to \R^d$ is quasiregular, then $x \in B_f$ if and only if $i(x, f) \geq 2$.
\subsection{The capacity of a condenser}
An important tool in the theory of quasiregular maps is the capacity of a condenser, and we recall this concept very briefly. If $A \subset \R^d$ is an open set, and $C \subset A$ is non-empty and compact, then the pair $(A,C)$ is called a \emph{condenser}. Its \emph{capacity}, which we denote by cap$(A,C)$, is defined by
$$
\operatorname{cap}(A,C) := \inf_u \int_A |\nabla u|^d dm.
$$
Here the infimum is taken over all non-negative functions $u \in C^\infty_0(A)$ that satisfy $u(x) \geq 1$, for $x \in C$.

If cap$(A,C) = 0$ for some bounded open set $A$ containing $C$, then cap$(A',C) = 0$ for every bounded open set $A'$ containing $C$; see \cite[Lemma III.2.2]{MR1238941}. In this case we say that $C$ \emph{has zero capacity}, and write cap $C = 0$; otherwise we say that $C$ \emph{has positive capacity}, and write cap~$C > 0$. For an unbounded closed set $C\subset\R^d$, we say that $C$ has zero capacity if every compact subset of $C$ has zero capacity. Roughly speaking, cap $C = 0$ means that $C$ is a ``small'' set.

The following \cite[Theorem II.10.10]{MR1238941} gives a link between the capacity of a condenser and quasiregular maps.
\begin{theorem}
\label{theo:condenser}
Suppose that $f : \R^d \to \R^d$ is a non-constant quasiregular map, and that $(A, C)$ is a condenser. Then
$$
\operatorname{cap }(f(A), f(C)) \leq \frac{K_I(f)}{m(f, C)} \operatorname{cap }(A, C),
$$
where
$$
m(f, C) := \inf_{y \in f(C)} \sum_{x \in C \cap f^{-1}(y)} i(x, f).
$$
\end{theorem}

If $(A,C)$ is a condenser in $\R^d$ such that both $C$ and $\R^d \setminus A$ are connected, then we say that it is \emph{ringlike}. We use the following estimate on the capacity of a ringlike condenser \cite[Lemma 7.34]{MR950174}; this is closely related to the Loewner property for $\R^d$ (see \cite[Definition~3.1]{MR1654771}).
\begin{lemma}
\label{Lemma7.34}
Suppose that $(A,C)$ is a ringlike condenser and that $A$ is bounded. Suppose also that $a, b \in C$ are distinct, and that $c \notin A$. Then
$$
\operatorname{cap} (A, C) \geq \tau_d\left(\frac{|a-c|}{|a-b|}\right).
$$
\end{lemma}
Here $\tau_d : (0, \infty) \to (0, \infty)$ is the capacity of the $d$-dimensional \emph{Teichm\"uller condenser}. We refer to, for example, \cite[Section 7]{MR950174}, for more information regarding this function. The only property of $\tau_d$ that we require is the following lower bound \cite[Equation 7.24]{MR950174}.
\begin{lemma}
\label{Theorem7.263}
There are positive constants $s_d$ and $t_d$, which depend only on $d$, such that
$$
\tau_d(r) \geq s_d(\log(t_d (r+1)))^{1-d}, \qfor r > 0.
$$
\end{lemma}

\subsection{The conformal invariant $\mu$}
As in \cite{SixsmithNicks1, SixsmithNicks2}, we use a conformal invariant that has proved useful when working with quasiregular maps. Since the definition of this invariant is based on auxiliary definitions which are not important for what follows, we merely state the properties we use, and refer to \cite{MR950174} for full details. Let $G \subset \R^d$ be a domain. Vuorinen \cite[p.103]{MR950174} defines a function
$$
  \mu_G : G \times G \to \R
$$
with the properties that $\mu_G$ is a conformal invariant, and is a metric if cap $\partial G > 0$. It is noted that if $D \subset G$ is a domain, then
\begin{equation}
\label{smallerdomain}
\mu_D(x,y) \geq \mu_G(x,y), \qfor  x, y \in D.
\end{equation}

We use the following bound on the change in this metric under a quasiregular map \cite[Theorem 10.18]{MR950174}. 
\begin{lemma}
\label{l3}
Suppose that $f : G \to \mathbb{R}^d$ is a non-constant quasiregular map. Then
$$
\mu_{f(G)}(f(a), f(b)) \leq K_I(f) \mu_G(a, b), \qfor a, b \in G.
$$
\end{lemma}
We use the following estimate; see \cite[Proposition 3]{MR752460}, which follows directly from \cite[Lemma 5.9]{MR0259114}.
\begin{lemma}
\label{Mohri}
Suppose that $G \subset \R^d$ is a domain such that meas$(G) < \infty$. Then there is a positive constant $b_d$, which depends only on $d$, such that
$$
\mu_G(x, y)^{d-1} \geq b_d \frac{|x - y|^d}{\operatorname{meas}(G)}, \qfor x, y \in G. 
$$
\end{lemma}
\subsection{The maximum modulus function and the fast escaping set}
We use the following property of the iterated maximum modulus \cite[Corollary 2.2]{SixsmithNicks1}.
\begin{lemma}
\label{Blemm}
Suppose that $f : \mathbb{R}^d \to \mathbb{R}^d$ is a quasiregular function of transcendental type. If $R>0$ is sufficiently large that $M^k(R, f)\rightarrow\infty$ as $k\rightarrow\infty$, then
\begin{equation}
\label{Beq}
\lim_{k\rightarrow\infty} \frac{1}{k} \log \log M^k(R, f) = \infty.
\end{equation}
\end{lemma}

The \emph{fast escaping set} is a subset of the \emph{escaping set}, which is the set of points which tend to infinity on iteration. It is defined by
\begin{equation}
\label{Adef}
A(f) := \{x : \text{there exists } \ell \isnatural \text{ s.t. } |f^{k+\ell}(x)| \geq M^k(R, f), \text{ for } k \isnatural\}.
\end{equation}
Here $R > 0$ can be taken to be any value such that $M^k(R,f)\rightarrow\infty$ as $k\rightarrow\infty$. When $f$ is a quasiregular map of transcendental type, this definition was first used in \cite{MR3215194}, and it was shown there that $A(f)$ is independent of the choice of $R$.

\subsection{Quasi-Fatou components}
We use three results from \cite{SixsmithNicks1}, which concern the properties of quasi-Fatou components. The first is part of \cite[Theorem 1.3]{SixsmithNicks1}, the second is \cite[Corollary 5.2]{SixsmithNicks1} and the third is \cite[Theorem 1.4]{SixsmithNicks1}.
\begin{lemma}
\label{Lboundednottopconv}
Suppose that $f : \mathbb{R}^d \to \mathbb{R}^d$ is a quasiregular function of transcendental type and that $U_0$ is a quasi-Fatou component of $f$ which is bounded and hollow. Then each $U_k$ is bounded and hollow, and $\overline{U_k} \subset A(f)$. Furthermore $U_k$ is contained in a bounded complementary component of $U_{k+1}$ for all large $k$, and dist$(0, U_k) \rightarrow\infty$ as $k\rightarrow\infty$.
\end{lemma}
\begin{lemma}
\label{lemfull}
Suppose that $f : \mathbb{R}^d \to \mathbb{R}^d$ is a quasiregular function of transcendental type. Suppose that $U$ is a quasi-Fatou component of $f$, and $V$ is the quasi-Fatou component of $f$ containing $f(U)$. Then $U$ is full if and only if $V$ is full.
\end{lemma}
\begin{lemma}
\label{Lunboundednottopconv}
Suppose that $f : \mathbb{R}^d \to \mathbb{R}^d$ is a quasiregular function of transcendental type. Suppose that $U$ is a quasi-Fatou component of $f$ which is unbounded and hollow. Then $U$ is completely invariant, its boundary components are bounded, and all other quasi-Fatou components of $f$ are full.
\end{lemma}
%
%
%
\section{Preimage components under a continuous, open, discrete map}
\label{S.topology}
In this section we give a number of preliminary results that are mostly topological in nature. The following is the main result of this section. Although this may not be entirely new, we have not been able to find a reference, apart from part of (\ref{th3.1:easythings}), which is \cite[Lemma I.4.7]{MR1238941}. If $f$ is an open map, $U \subset \R^d$ is a bounded domain, and $\partial f(U) = f(\partial U)$, then we say that $U$ is a \emph{normal domain} of $f$.
\begin{theorem}
\label{theo:boundedpreimage}
Suppose that $f : \R^d \to \R^d$ is a continuous, open and discrete map. Suppose that $U'$ is a domain, and $U$ is bounded component of $f^{-1}(U')$. Then:
\begin{enumerate}[(i)]
\item $U'$ is bounded, $f(U) = U'$, $f(\partial U) = \partial U'$, $U$ is a normal domain of $f$, $f(\partial_{out} U) = \partial_{out} U'$, and finally $f(\hull(U)) = \hull(U')$;\label{th3.1:easythings}
\item if $V$ is a bounded component of the complement of $U$, then $f(V)$ is a bounded component of the complement of $U'$;\label{th3.1:allofboundedcomponent} and 
\item we have $\operatorname{cc}(U') \leq$ $\operatorname{cc}(U)$, and if $\operatorname{cc}(U)=\infty$, then $\operatorname{cc}(U')=\infty$.\label{th3.1:ccdecreases}
\end{enumerate}
\end{theorem}
To prove this result we need a number of preliminary results. The first is \cite[Proposition 2.4]{MR3215194}.
\begin{proposition}
\label{prop:BDF}
Suppose that $f : \R^d \to \R^d$ is a continuous open map, and that $U \subset \R^d$ is a bounded open set. Then 
\begin{equation}
\label{BDFeq}
f(\hull(U)) \subset \hull(f(U)) \quad\text{and}\quad \partial \hull(f(U)) \subset f(\partial \hull(U)).
\end{equation}
\end{proposition}
The second is \cite[Lemma I.4.8]{MR1238941}.
\begin{proposition}
\label{prop:Rickman}
Suppose that $f : \R^d \to \R^d$ is a continuous, open and discrete map, that $U \subset \R^d$ is a normal domain of $f$, and that $E \subset f(U)$ is a continuum. Then $f$ maps every component of $U \cap f^{-1}(E)$ onto $E$.
\end{proposition}
We also use the following elementary proposition.
\begin{proposition}
\label{fullequalscc1}
Suppose that $U$ is a proper subdomain of $\R^d$. Then $U$ is full if and only if $\operatorname{cc}(U)= 1$.
\end{proposition}
\begin{proof}
Recall that $\operatorname{cc}(U)$ denotes the number of components of $\widehat{\R^d} \setminus U$. First suppose that $\operatorname{cc}(U) > 1$. Then $\widehat{\R^d} \setminus U$ has a bounded component, and so $U$ is hollow. 

On the other hand, suppose instead that $U$ is hollow, and so $\R^d \setminus U$ has a bounded component, $E$ say. As in the proof of \cite[Theorem 1.4]{SixsmithNicks1}, there is a bounded domain, $G$, such that $E \subset G$ and $\partial G \subset U$. It follows that $E$ and $\infty$ are in different components of $\widehat{\R^d} \setminus U$. Hence  $\widehat{\R^d} \setminus U$ is not connected, and so $\operatorname{cc}(U) > 1$.
\end{proof}
We can now prove Theorem~\ref{theo:boundedpreimage}.
\begin{proof}[Proof of Theorem~\ref{theo:boundedpreimage}]
Suppose that $f : \R^d \to \R^d$ is a continuous, open and discrete map, that $U'$ is a domain, and that $U$ is bounded component of $f^{-1}(U')$. We prove the three parts of Theorem~\ref{theo:boundedpreimage} in order.

First we prove (\ref{th3.1:easythings}). Since $f$ is an open map and $U$ is bounded, $\partial f(U) \subset f(\partial U)$. Since $f$ is continuous we also have
\begin{equation}
\label{partialUeq1}
f(\partial U) \subset f(\overline{U}) \subset \overline{f(U)}.
\end{equation}

Suppose that $x \in \partial U$. If $f(x) \in U'$, then $x$ lies in a component of $f^{-1}(U')$, contradicting the fact that $x$ lies on the boundary of a component of $f^{-1}(U')$. Hence $f(\partial U) \cap U' = \emptyset$ and, in particular, $f(\partial U) \cap f(U) = \emptyset$.

We deduce from (\ref{partialUeq1}) that
$$
f(\partial U) \subset \overline{f(U)}\backslash f(U) = \partial f(U),
$$
and it follows that $\partial f(U) = f(\partial U)$, and so $U$ is a normal domain for $f$. Clearly $f(U) \subset U'$. If $f(U) \ne U'$, then there exists $x \in \partial f(U) \cap U' = f(\partial U) \cap U' = \emptyset$. We deduce that $f(U) = U'$, that $f(\partial U) = \partial U'$, and that $U'$ is bounded.

Next, we note by Proposition~\ref{prop:BDF} that
$$
f(\partial_{out} U) = f(\partial \hull(U)) \supset \partial \hull(f(U)) = \partial \hull(U') = \partial_{out} U'.
$$
Since $f(\partial_{out} U)$ is connected and contained in $\partial U'$, we conclude that $f(\partial_{out} U)~=~\partial_{out} U'$.

Since $f(\hull(U)) \subset \hull(U')$ by Proposition~\ref{prop:BDF}, and 
$$
\partial f(\hull(U)) \subset f(\partial \hull(U)) = f(\partial_{out} U) = \partial_{out} U',
$$
we conclude that $f(\hull(U)) = \hull(U')$. This completes the proof of (\ref{th3.1:easythings}). \\

For the proof of (\ref{th3.1:allofboundedcomponent}), let $V$ be a bounded component of the complement of $U$. Since $\partial f(V) \subset f(\partial V)$ and $f(\partial V) \cap U' = \emptyset$, it follows that $\partial f(V) \cap U' = \emptyset$. Also, by Proposition~\ref{prop:BDF} and by (\ref{th3.1:easythings}),
$$
f(V) \subset f(\hull(U)) \subset \hull(f(U)) = \hull(U').
$$
Since $f(V)$ is closed, it follows that $f(V)$ does not contain $U'$. Furthermore, since $f(\partial U) \cap U' = \emptyset$, we have  $f(V) \cap U' = \emptyset$. Hence $f(V)$ is contained in a bounded component of the complement of $U'$, say $V'$.

Note that $V' \subset \hull(U')$ is a continuum, and also that $V$ is a component of $\hull(U) \cap f^{-1}(V')$. Hence, by Proposition~\ref{prop:Rickman} and (\ref{th3.1:easythings}), $f(V) = V'$, as required. \\

To prove the first part of (\ref{th3.1:ccdecreases}), suppose that $V'$ is a bounded complementary component of $U'$. It follows from (\ref{th3.1:easythings}) and (\ref{th3.1:allofboundedcomponent}) that there is a bounded complementary component of $U$, say $V$, such that $V' = f(V)$. We note that $U$ is bounded and so $\operatorname{cc}(U)$ is one more than the number of bounded complementary components of $U$. The analogous remark holds for $U'$. The result follows.

For the second part of (\ref{th3.1:ccdecreases}), suppose, by way of contradiction, that $\operatorname{cc}(U)=\infty$ and that $\operatorname{cc}(U')$ is finite. It follows from (\ref{th3.1:easythings}) and (\ref{th3.1:allofboundedcomponent}) that there is a bounded component of the complement of $U'$, say $V'$, such that $V'$ is the image of infinitely many bounded components of the complement of $U$. Choose $y \in V'$. It follows by (\ref{th3.1:allofboundedcomponent}) that $f^{-1}(y) \cap \hull(U)$ is infinite and bounded, and hence has an accumulation point in $\R^d$, which is impossible since $f$ is discrete.
\end{proof}
The application of Theorem~\ref{theo:boundedpreimage} in our context is achieved using the following.
\begin{lemma}
\label{lemm:1a}
Suppose that $f : \R^d \to \R^d$ is quasiregular, that $U$ is a bounded quasi-Fatou component of $f$ and that $U'$ is the quasi-Fatou component of $f$ containing $f(U)$. Then $U$ is a bounded component of $f^{-1}(U')$.
\end{lemma}
\begin{proof}
Clearly $U$ is contained in a component, $W$ say, of $f^{-1}(U')$. If $U \ne W$, then $\partial U \cap W \ne \emptyset$. This is a contradiction, since $\partial U \subset J(f)$ and $W \subset QF(f)$.
\end{proof}
%
%
%
\section{Proof of Theorem~\ref{theo:boundedhollowQFsimple}}
\label{S:boundedhollowQFsimple}
We first prove the following, which is also used elsewhere in this paper.
\begin{lemma}
\label{lemm:boundedhollowQF}
Suppose that $f : \R^d \to \R^d$ is a quasiregular map of transcendental type, and that $U_0$ is a bounded, hollow quasi-Fatou component of $f$. 
Then, for $k\geq 0$,
\begin{equation}
\label{basic1}
U_k = f^k(U_0) \quad\text{and}\quad \hull(U_k) = f^k(\hull(U_0)).
\end{equation} 
Also, there exists $k'\isnatural$ such that,
\begin{equation}
\label{inb1}
f(\hullo(U_k)) = \hullo(U_{k+1}), \qfor k\geq k'.
\end{equation}
In addition 
\begin{equation}
\label{cc}
\operatorname{cc}(U_{k}) \leq \operatorname{cc}(U_{k-1}), \qfor k\isnatural.
\end{equation}
If $1<\operatorname{cc}(U_0)<\infty$, then $\operatorname{cc}(U_k)=2$, for all sufficiently large $k$. Finally, if $\operatorname{cc}(U_0)=\infty$, then $\operatorname{cc}(U_k)=\infty$, for $k\isnatural$.
\end{lemma}
\begin{proof}
Suppose that $f : \R^d \to \R^d$ is a quasiregular map of transcendental type, and that $U_0$ is a bounded, hollow quasi-Fatou component of $f$. Observe throughout this proof, by Lemma~\ref{lemm:1a}, that the hypotheses of Theorem~\ref{theo:boundedpreimage} are satisfied with $U = U_k$ and $U' = U_{k+1}$, for $k\geq 0$. Equation (\ref{basic1}) follows by repeated application of Theorem~\ref{theo:boundedpreimage} part (\ref{th3.1:easythings}). \\ 

Next, it follows from Lemma~\ref{Lboundednottopconv} that we can choose $k'$ sufficiently large that
$$
\{0, f(0)\} \subset \hullo(U_k) \quad\text{and}\quad U_k \subset \hullo(U_{k+1}), \qfor k \geq k'.
$$

Suppose that $k \geq k'$. Then $f(0) \in \hullo(U_{k+1})$ and $0 \in \hullo(U_k)$, the latter implying that also $f(0) \in  f(\hullo(U_k))$. Hence (\ref{inb1}) follows, by Theorem~\ref{theo:boundedpreimage} part (\ref{th3.1:allofboundedcomponent}). \\

Equation (\ref{cc}) is an immediate consequence of Theorem~\ref{theo:boundedpreimage} part (\ref{th3.1:ccdecreases}). \\

Now, let 
$$
V_k = \hull(U_k)\backslash \hullo(U_k), \qfor k\geq k'.
$$
We consider two cases. First, we suppose that there exists $k'' \geq k'$ such that $f(V_k) \subset V_{k+1}$, for $k\geq k''$. We deduce from the definition of $J(f)$ that $V_k \subset QF(f)$, for $k\geq k''$. Since $V_k$ is connected, $U_k \subset V_k$ and $U_k$ is a component of $QF(f)$, we deduce that $V_k = U_k$ and so $\operatorname{cc}(U_k) = 2$. It then follows by Theorem~\ref{theo:boundedpreimage} part (\ref{th3.1:ccdecreases}) that $1 < \operatorname{cc}(U_0) < \infty$.

Suppose, on the other hand, that $f(V_k) \not \subset V_{k+1}$ for infinitely many values of $k$. We complete the proof of the theorem by showing that this implies that $\operatorname{cc}(U_k) = \infty$, for $k \geq 0$. Suppose, by way of contradiction, that there exists $k'' \geq 0$ such that $\operatorname{cc}(U_{k''})$ is finite. It follows by Theorem~\ref{theo:boundedpreimage} part (\ref{th3.1:ccdecreases}) that $\operatorname{cc}(U_k)$ is finite, for $k\geq {k''}$. Note, by Proposition~\ref{prop:BDF}, that
$$
f(V_k) \subset f(\hull(U_k)) \subset \hull(f(U_k)) = \hull(U_{k+1}).
$$

Now, let $k\geq \max\{k',k''\}$ be such that $f(V_k) \not \subset V_{k+1}$. Hence there exists $y \in V_k$ such that $f(y) \in \hullo(U_{k+1})$. It follows that $y$ belongs to a bounded complementary component of $U_k$, say $Y$, with $Y \ne \hullo(U_k)$. We deduce, by Theorem~\ref{theo:boundedpreimage}, that $f(Y) = f(\hullo(U_k))$, and so $\operatorname{cc}(U_{k+1}) < $ $\operatorname{cc}(U_k)$. Since we have assumed that this happens for infinitely many values of $k$, we deduce a contradiction.
\end{proof}
\begin{proof}[Proof of Theorem~\ref{theo:boundedhollowQFsimple}] 
Suppose that $f : \R^d \to \R^d$ is a quasiregular map of transcendental type, and that $U_0$ is a quasi-Fatou component of $f$. 

Suppose that $\operatorname{cc}(U_k)=1$, for some $k\isnatural$. We deduce, by Proposition~\ref{fullequalscc1} and Lemma~\ref{lemfull}, that $\operatorname{cc}(U_k)=1$, for all $k\isnatural$. This deals with the first case.

Suppose next that $U_0$ is unbounded and that $\operatorname{cc}(U_0) \ne 1$, in which case $U_0$ is hollow by Proposition~\ref{fullequalscc1}. Recalling that $J(f)$ is unbounded, Lemma~\ref{Lunboundednottopconv} now implies that $\R^d \setminus U_0$ has infinitely many bounded components. Since $U_0$ is open, it follows that $\operatorname{cc}(U_0)=\infty$. The complete invariance part of Lemma~\ref{Lunboundednottopconv} tells us that $U_0=U_k$, for $k \geq 0$, and thus $\operatorname{cc}(U_0)=\operatorname{cc}(U_k)$.

There remains only the case that $U_0$ is bounded and hollow. It is easy to see that in this case Theorem~\ref{theo:boundedhollowQFsimple} is a consequence of Lemma~\ref{lemm:boundedhollowQF}.
\end{proof}
%
%
%
%
%
%
\section{Proofs of Theorem~\ref{theo:curvefambetterDan} and Theorem~\ref{theo:curvefambetter}}
\label{S:curvefambetterDan}
\begin{proof}[Proof of Theorem~\ref{theo:curvefambetterDan}]
Without loss of generality, it follows by Lemma~\ref{Lboundednottopconv} that we can assume the following. For each $k\geq 0$, the origin is contained in a bounded complementary component of the closure of $U_k$; we label this component $D_k$. Moreover, by Lemma~\ref{lemm:boundedhollowQF}, we can also assume that
\begin{equation}
\label{equation:things2}
f^{k}(\hullo(U_{0})) = \hullo(U_k), \text{ and } f^{k}(\hull(U_{0})) = \hull(U_k), \qfor k\geq 0.
\end{equation}

Suppose that there exist $k\geq 0$ and $x \in \partial D_k$ such that $f(x) \in \operatorname{int }f(D_k)$. Since int$f(D_k) \subset$ int$f(\hullo(U_k))$, it follows that $f(x)$ is in the interior of the complement of $U_{k+1}$. On the other hand, $x \in \partial D_k$ implies that $x \in \partial U_k$, which in turn implies that $f(x) \in \partial U_{k+1}$. This is a contradiction. It follows that, for each $k\geq 0$, $f(\partial D_k) \subset \partial f(D_k)$, and thus $D_k$ is a normal domain for $f$. It can be seen that we may also assume that $f(D_k) = D_{k+1}$, for $k\geq 0$. \\

Let $L > 0$ be large, and take an integer $p \geq 2LK_I(f)$. Since $f$ is of transcendental type, there exists $R>0$ such that $N(f, \{ x : |x| < R \}) \geq p$. By Lemma~\ref{Lboundednottopconv} we can choose $k_0\isnatural$ sufficiently large that, for $k \geq k_0$, we have $\{ x : |x| < R \} \subset D_k$ and hence $N(f, D_k) \geq p$.

It is known that if $D$ is a normal domain of $f$, and $y \in f(D)$, then 
\begin{equation}
\label{eqN2}
N(f, D) = \sum_{x \in D \cap f^{-1}(y)} i(x, f).
\end{equation}
(This follows, for example, from \cite[Lemma 9.15]{MR950174} together with the comments on \cite[p.123]{MR950174}.) We note that $D_0$ is a normal domain for $f^k$. 

Suppose that $k \geq k_0$. We claim that
\begin{equation}
\label{Nclaim}
N(f^k, D_0) = \prod_{i=1}^k N(f, D_{k-i}) \geq p^{k-k_0}.
\end{equation}
We establish this claim as follows. Choose $y \in D_k$ such that
\begin{equation}
\label{ychoice}
B_f \cap \bigcup_{j=1}^k\left(f^{-j}(y) \cap D_{k-j}\right) = \emptyset.
\end{equation}
This choice is possible by Lemma~\ref{lemma:branch}. We then observe, by (\ref{eqN2}) and (\ref{ychoice}), that
$$
\operatorname{card } \left(f^{-1}(y) \cap D_{k-1}\right) = N(f, D_{k-1}),
$$
and then, by induction, that
\begin{equation}
\label{preims}
\operatorname{card } \left(f^{-j}(y) \cap D_{k-j}\right) = \prod_{i=1}^j N(f, D_{k-i}), \qfor 1 \leq j \leq k.
\end{equation}
Our claim then follows from (\ref{eqN2}), (\ref{ychoice}), (\ref{preims}) with $j=k$, and from the choice of $k_0$. \\

We aim next to show that
\begin{equation}
\label{dannew}
m(f^k, \hullo(U_0)) \geq N(f^k, D_0), \qfor k \geq 0.
\end{equation}

To prove this, fix $k\geq 0$ and let $W := \{ x \in \R^d : \operatorname{dist}(x, \hullo(U_k)) < \delta \}$, where we choose $\delta>0$ to be sufficiently small that the component $V$ of $f^{-k}(W)$ that contains $\hullo(U_0)$ meets no other component of $f^{-k}(\hullo(U_k))$. It is straightforward to see that this is possible, by (\ref{equation:things2}) and since there are only finitely many components of $f^{-k}(\hullo(U_k))$ in the bounded set $\hull(U_0)$.

Since $V \subset \hull(U_0)$, we see that $V$ is bounded. Hence, by Theorem~\ref{theo:boundedpreimage} part (\ref{th3.1:easythings}), $V$ is a normal domain for $f^k$. Since $D_0 \subset \hullo(U_0) \subset V$, it follows by (\ref{eqN2}) that, for each $y \in f^k(\hullo(U_0)) = \hullo(U_k)$, we have
$$
\sum_{x \in \hullo(U_0) \cap f^{-k}(y)} i(x, f^k) = \sum_{x \in V \cap f^{-k}(y)} i(x, f^k) = N(f^k, V) \geq N(f^k, D_0).
$$
We obtain (\ref{dannew}) by taking the infimum over all $y \in f^k(\hullo(U_0))$.\\

We deduce from Theorem~\ref{theo:condenser}, together with (\ref{Keq}), (\ref{equation:things2}), (\ref{Nclaim}) and (\ref{dannew}) that, for all sufficiently large values of $k$,
\begin{align}
\operatorname{cap }(\hull(U_k), \hullo(U_k)) &\leq \frac{K_I(f^k)}{m(f^k, \hullo(U_0))} \operatorname{cap }(\hull(U_0), \hullo(U_0)) \nonumber \\
                                      &\leq \frac{K_I(f)^k}{p^{k-k_0}} \operatorname{cap }(\hull(U_0), \hullo(U_0)) \leq \frac{1}{L^{k}}.\label{eq:Mk2}
\end{align}

Next we fix values $a,b,c \in \R^d$ for use in Lemma~\ref{Lemma7.34}. We let $a = 0 \in \hullo(U_k)$, let $b\in \hullo(U_k)$ be such that $$|b| = \sup \{|x| : x \in \hullo(U_k) \} = \sup \{|x| : x \in \partial_{inn} U_k \},$$ and let $c\in \R^d \setminus \hull(U_k)$ be such that $$|c| = \inf \{|x| : x \in \R^d \setminus \hull(U_k) \} = \inf \{|x| : x \in \partial_{out} U_k \}.$$
It follows by Lemma~\ref{Lemma7.34} that
$$
\operatorname{cap }(\hull(U_k), \hullo(U_k)) \geq \tau_d\left(\frac{|c|}{|b|}\right).
$$

We deduce by (\ref{eq:Mk2}) and Lemma~\ref{Theorem7.263} that there are constants $s_d, t_d>0$, such that for all sufficiently large values of $k$,
$$
\frac{1}{L^{k}} \geq \operatorname{cap }(\hull(U_k), \hullo(U_k)) \geq \tau_d\left(\frac{|c|}{|b|}\right) \geq s_d\left(\log\left(t_d\left(\frac{|c|}{|b|}+1\right)\right)\right)^{1-d}.
$$

The theorem follows, by a calculation, since $L$ was arbitrary.
\end{proof}
Theorem~\ref{theo:curvefambetter} is now a relatively straightforward consequence of Theorem~\ref{theo:curvefambetterDan}.
\begin{proof}[Proof of Theorem~\ref{theo:curvefambetter}]
Suppose that $f : \R^d \to \R^d$ is a quasiregular map of transcendental type, and that $U_0$ is a hollow quasi-Fatou component of $f$ with finitely many complementary components. If $U_0$ were unbounded, then, by Lemma~\ref{Lunboundednottopconv}, its complement would be bounded. However, this complement contains the unbounded set $J(f)$. It follows that $U_0$ is bounded. 

Without loss of generality, by Lemma~\ref{Lboundednottopconv}, we can set
$$
R_k := \inf \{ |x| : x \in \partial_{out} U_k \} \quad\text{and}\quad r_k := \sup \{|x| : x \in \partial_{inn} U_k\}, \qfor k\isnatural.
$$

We deduce part (\ref{theointoinf}) of the theorem from Lemma~\ref{Lboundednottopconv}, and part (\ref{theofracsgrow}) of the theorem from Theorem~\ref{theo:curvefambetterDan}. It follows from Theorem~\ref{theo:boundedhollowQFsimple} that, for sufficiently large values of $k$, the domain $U_k$ has exactly two complementary components. Hence, for all large values of $k$, we have that $\{ x : r_k < |x| < R_k \} \subset U_k$, as required. 
\end{proof}
%
%
%
\section{Proof of Theorem~\ref{theo:measureofboundedhollowQF}}
\label{S:measureofboundedhollowQF}
\begin{proof}[Proof of Theorem~\ref{theo:measureofboundedhollowQF}]
First we require some preliminaries. It follows by Lemma~\ref{Lboundednottopconv} that we can choose $k_0 \in \N$ sufficiently large that, for each $k\geq k_0$, the origin is contained in a bounded complementary component of $U_k$, and that $U_k \subset \hullo(U_{k+1})$. We can also assume that (\ref{inb1}) holds for $k\geq k_0$.

Let $W$ be a bounded domain such that $\hullo(U_{k_0})\subset W$ and $\partial W \subset U_{k_0}$; this choice is possible since $\hullo(U_{k_0})$ is a bounded component of the complement of $U_{k_0}$. We claim that, for all $k\geq {k_0}$, the origin is contained in a bounded complementary component of $f^{k-k_0}(\partial W)$. For, by (\ref{inb1}) and our initial assumption, 
$$
0 \in \hullo(U_k) = f^{k-k_0}(\hullo(U_{k_0})) \subset f^{k-k_0}(W), \qfor k\geq {k_0}.
$$

Moreover, by our choice of $W$ and initial assumption, $f^{k-k_0}(\partial W) \subset U_{k}$ and $0 \notin U_{k}$. Hence $0\notin f^{k-k_0}(\partial W)$. It follows that $0$ is contained in a bounded component of the complement of $f^{k-k_0}(\partial W)$. This completes the proof of our claim.

We are now ready to prove the theorem. Fix $v \in U_{k_0}$, and set
$$
L = L(v) = \sup_{x \in \partial W} \mu_{U_{k_0}}(v,x).
$$
Note that the topology given by the metric $\mu_{U_{k_0}}$ is the Euclidean topology on $U_{k_0}$; see \cite[Theorem 1]{MR752460} and also \cite[Theorem 2]{MR0302904}. Hence it follows from the fact that $\partial W$ is a compact subset of $U_{k_0}$ that $L$ is finite.

Suppose that $k\geq {k_0}$. Since the origin is contained in a bounded complementary component of $f^{k-k_0}(\partial W)$, we can choose $x_k \in \partial W$ such that 
$$
| f^{k-k_0}(v) - f^{k-k_0}(x_k) | \geq |f^{k-k_0}(v)|.
$$
For example, we can choose $x_k$ such that the points $f^{k-k_0}(v)$, $0$ and $f^{k-k_0}(x_k)$ are colinear, with $0$ separating the other two points. It follows by (\ref{Keq}), Lemma~\ref{l3} and Lemma~\ref{lemm:boundedhollowQF}, that
\begin{align}
L K_I(f)^{k-k_0} &\geq K_I(f)^{{k-k_0}} \mu_{U_{k_0}}(v,x_k) \nonumber \\
                 &\geq K_I(f^{k-k_0}) \mu_{U_{k_0}}(v,x_k) \nonumber \\
                 &\geq \mu_{f^{k-k_0}({U_{k_0}})}(f^{k-k_0}(v),f^{k-k_0}(x_k)) \nonumber \\
                 &\geq \mu_{U_k}(f^{k-k_0}(v),f^{k-k_0}(x_k)).
\end{align}

We obtain, by Lemma~\ref{Mohri}, that there exists a constant $b_d > 0$ such that
$$
L^{d-1} K_I(f)^{(k-k_0)(d-1)} \geq b_d \frac{|f^{k-k_0}(v) - f^{k-k_0}(x_k)|^d}{\operatorname{meas}(U_k)} \geq b_d \frac{|f^{k-k_0}(v)|^d}{\operatorname{meas}(U_k)}.
$$
It follows that
\begin{equation}
\label{1.6e1}
\operatorname{meas}(U_k) \geq \frac{b_d}{L^{d-1}} K_I(f)^{(k-k_0)(1-d)}|f^{k-k_0}(v)|^d.
\end{equation}

We note, by Lemma~\ref{Lboundednottopconv}, that $v \in A(f)$. It follows, by (\ref{Adef}), that there exists $\ell \isnatural$ such that \begin{equation}
\label{1.6e2}
|f^{k-k_0+\ell}(v)| \geq M^{k}(R, f), \qfor k \geq k_0.
\end{equation}

Equation (\ref{eq:tnewer}) follows from (\ref{1.6e1}) and (\ref{1.6e2}). In addition, we can assume that $R$ is sufficiently large that (\ref{Beq}) holds. Equation (\ref{eq:tnew}) follows.
\end{proof}
%
%
%
%
%
\emph{Acknowledgment:} The authors are grateful to Phil Rippon and Gwyneth Stallard for asking about a generalisation of Theorem C. The authors are also grateful to the referee for many helpful remarks and suggestions.
%
%
%
%
%
%

\end{document}